\documentclass{amsart}
\usepackage[english]{babel}
\usepackage[latin1]{inputenc}
\usepackage{amsfonts}
\usepackage{amssymb}
\usepackage{amsmath}
\usepackage{amsthm}
\usepackage{graphicx}
\usepackage{epsfig}
\usepackage{fancyhdr}
\usepackage[all]{xy}

\numberwithin{equation}{section}

\newcommand{\C}{\mathbb{C}}
\newcommand{\R}{\mathbb{R}}
\newcommand{\N}{\mathbb{N}}
\newcommand{\Z}{\mathbb{Z}}

\newtheorem{theorem}{Theorem}[section]
\newtheorem{lemma}[theorem]{Lemma}

\newtheorem{cor}[theorem]{Corollary}

\newtheorem{defi}[theorem]{Definition}

\usepackage{amsmath}

\begin{document}
\title{Remling's Theorem on canonical systems }
 \maketitle
 \begin{center}\author{KESHAV RAJ ACHARYA }\end{center}

\begin{abstract} We prove the Remling's Theorem on canonical systems and discuss the connection between Jacobi and Schrödinger equation and canonical systems.

  \end{abstract}

\section{Introduction}This paper deals with the canonical system of the following form \begin{equation}\label{ca}Ju'(x)=z H(x)u(x).\end{equation} Here $ J =
\left(\begin{array}{cc}0&-1\\1&0\end{array}\right)$ and $H(x)$ is a $2\times2$ positive semidefinite matrix whose entries  are locally integrable and that there is no non-empty open interval  $I$ so that $H = 0$  $a.e.$ on $I$.
The complex number $z\in\C$ involved in \ref{ca} is a spectral parameter. For fixed $z\in \C ,$ a function
$u(.,z):[0,N]\rightarrow \C^2$ is called a solution if $u$ is
absolutely continuous and satisfies \ref{ca}. Consider the Hilbert space
\[\displaystyle L^2(H,\R_+) = \Big \{f(x)
=\left(\begin{array}{c}f_1(x)\\f_2(x)\end{array}\right) :
 \|f\| < \infty \Big \}\]
 with an inner product $\displaystyle\big<f, g\big> = \int_{0}^{\infty}f(x)^* H(x)g(x)dx.$
 Such canonical systems \ref{ca} on $L^2(H,\R_+)$ have been studied by Hassi, De snoo,  Winkler, and  Remling  in the papers [  \cite{HW2},\cite{ RC}, \cite{HW},  \cite{cr} ] in various context. The Jacobi and Schrödinger equations can be written into canonical systems with appropriate choice of $H(x)$. In addition, the canonical systems  are closely connected with the theory of de Branges spaces and the inverse spectral theory of one dimensional Schrödinger equations,  see \cite{RC}.\\

Recently, in the spectral theory of Jacobi and Schrödinger operators, the Remling's theorem, published in the \emph{Annals of Math} in 2011 (see \cite{RC1}), has been one of the most popular results. It has revealed some new fundamental properties of absolutely continuous spectrum of Jacobi and Schrödinger operators that changed the perspective of many mathematicians about the  absolutely continuous spectrum. In this paper we will prove
the Remling's Theorem  on canonical systems. \\

 This paper has been organized as follows: In section \ref{Weylth}, we discuss the Weyl theory of canonical systems following the analogous treatment of Weyl theory of Jacobi and Schrödinger equations. In section \ref{topo} we discuss the basic definitions and space of Hamiltonians in order to state the main theorem. In section \ref{mainth} we prove our main theorem using the similar techniques from \cite{RC1}, more specifically we  prove the Breimesser-Pearson theorem on canonical systems which is in fact the foundation for the proof of the Remling's theorem. Finally we show the connection between Jacobi and Schrödinger equations with canonical systems in section  \ref{jsc}.

 \section{Weyl Theory} \label{Weylth} Let $ u_{\alpha},v_{\alpha}$ be
the solution of \ref{ca} with the initial values \begin{align*}
u_{\alpha}(0,z)=\left(\begin{array}{c}\cos{\alpha}\\-\sin{\alpha}\end{array}\right)  \,\,\,\,
\txt{   and   }\,\,\,\,
    v_{\alpha}(0,z)=\left(\begin{array}{c}\sin{\alpha}\\\cos{\alpha}\end{array}\right), \,\,\,\,\alpha \in (0,\pi].\end{align*}
    For $z\in\C^+$, want to define $m_{\alpha}(z)\in\C $ as the
    unique coefficient for which\[ f_{\alpha}=
    u_{\alpha}+m_{\alpha}(z)v_{\alpha}\in L^2(H,\R_+).\] Consider a compact interval $[0,N]$ and let
    $z\in\C^+$, define the unique coefficient $m_{N}^{\beta}(z)$ as follows,  $ f(x,z)=u(x,z)+ m_{N}^{\beta}(z)v(x,z)$ satisfying
    \[f_1(N,z)\sin\beta+f_2(N,z)\cos\beta=0.\] Clearly this is well defined
   because $u(x,z)$ does not satisfies the boundary condition at
    $N.$ Otherwise $z \in C^+$ will be an eigen value for some self-adjoint relation of the system \ref{ca} as explained in \cite{KA}. From the boundary condition $f_1(N,z)\sin\beta+f_2(N,z)\cos\beta=0$ at $N$ we
    get\[m_N^{\beta}(z)= -\frac{u_1(N,z)\sin \beta+u_2(N,z)\cos \beta}{v_1(N,z)\sin \beta +v_2(N,z)\cos
    \beta}.\]As  $ z, N, \beta $  varies $m_N^{\beta}(z)$ becomes a funtion of these arguments, and since $u_1, u_2, v_1, v_2 $ are entire function of $z$ it follows that $m_N^{\beta}(z)$ is meromorphic function of $z.$ \\ Let  $C_N(z)=\{m_N^{\beta}(z): 0\leq\beta <
    \pi\}$\\
 Here
\begin{align*}m_N ^{\beta}(z)=- \frac{u_1t+u_2}{v_1t+v_2},
   \,\,\, t=\tan{\beta},\,\,\, t\in\R\cup\{\infty\}.\end{align*}This is a
    fractional linear transformation. As a function of  $ t\in \R$ it maps real line  to a circle.
    So for fixed $z\in C^+,$   $C_N(z)$ is a circle. Hence for any complex number  $m\in\C$ \[ m\in C_N(z) \Leftrightarrow \text{ Im }\frac{u_2+mv_2}{u_1+mv_1}=
    0\]
From this identity the equation of the circle $C_N(z)$ is given by
$ |m-c|^2= r^2$ where \begin{align} \label{cr} c =
\frac{W_N(u,\bar{v})}{W_N(\bar{v},v)},   \hspace{1.5cm}
r=\frac{1}{|W_N(\bar{v},v)|}\end{align}

Now suppose $ f(x,z)=u(x,z)+ m_{N}^{\beta}(z)v(x,z)$, then  $ m = m_{N}^{\beta}$ is an interior of $ C_N$ if and only if  \begin{align} \label{wc} |m-c|^2 < r^2
\Leftrightarrow \frac{W_N(\bar{f},f)}{W_N(\bar{v},v)}<0
\end{align}

 Let us write $\tau y = zy $ if and only if $ Jy' = zH(x)y$. Suppose $f$ and $g$ are the
solutions of \ref{ca} then we have the following identity, called the {\em Green's
Identity}. \begin{align} \int_0^{N}(f^*H(x)\tau g -
(\tau f)^*H(x)g(x))dx = W_0(\bar{f},g)-W_N(\bar{f},g)\end{align}

Using the Green's identity we have,

\begin{align}\label{gi} W_N(\bar{f},f)= 2i \text{ Im } m(z)-2i\text{ Im}z\int_0^Nf^*(x)
H(x)f(x)dx.\end{align}

\begin{align*} W_N(\bar{v},v)= - 2i\text{ Im}z\int_0^Nv^*(x)
H(x)v(x)dx.\end{align*}

\begin{align*}\frac{W_N(\bar{f},f)}{W_N(\bar{v},v)}  = \frac{- \text{ Im } m(z)+ \text{ Im}z\int_0^Nf^*(x)
H(x)f(x)dx }{\text{ Im}z \int_0^Nv^*(x)
H(x)v(x)dx}.\end{align*} Hence from \ref{wc} we see that $\frac{W_N(\bar{f},f)}{W_N(\bar{v},v)}<0 $  if and only if
 \begin{align*} \int_0^Nf^*(x) H(x)f(x)dx < \frac{\text {
Im} m(z)}{\text{ Im} z}.\end{align*}  Thus it follows that $ m$ is an interior of $C_N $ if and only if
\begin{align} \label{inc}
\int_0^Nf^*(x) H(x)f(x)dx < \frac{\text {
Im} m(z)}{\text{ Im} z} \end{align} and $ m\in C_N(z) $ if and only if

\begin{align} \label{mh}  \int_0^Nf^*(x)
H(x)f(x)dx =\frac{\text{ Im } m(z)}{\text{ Im } z}.\end{align}

For $z\in \C^+$, the radius of the circle $C_N(z)$ is given by
\begin{align} \label{radi} r_N(z) =
\frac{1}{|W_N(\bar{v},v)|}= \frac{1}{2 \text { Im } z\int_0^Nv^*(x)
H(x)v(x)dx} .\end{align}

Now let $0<N_1<N_2 <\infty$. Then  if $ m $ is inside or on $C_{N_2}$ \[ \displaystyle \int _0^{N_1} f^*(x,z) H(x) f(x,z)dx < \int _0^{N_2} f(x,z)^* H(x) f(x,z)dx \leq \frac{ \text{ Im }m}{ \text{ Im }z}\] and therefore $m$ is inside $C_{N_1}$. Let us denote the interior of $ C_N(z) $ by $\text{ Int} C_N(z) $  and suppose $D_N(z) = C_N(z)\cup \text{ Int} C_N(z).$  Then \begin{align*} m \in
D_N(z) \Leftrightarrow \int_0^Nf^*(x) H(x)f(x)dx \leq \frac{ \text{
Im }m(z)}{\text{ Im }z}.\end{align*}
These are called the Wyle Disks.
These Wyle Disks are nested. That is $D_{N+\epsilon}(z)\subset
D_N(z)$ for any $\epsilon>0$.
From \ref{radi} we see that $r_N(z)>0$, and $r_N(z)$ decreases  as
$  N\rightarrow \infty.$ So $ \displaystyle \lim_{N\rightarrow \infty} r_N(z)$
exists and \[  \lim_{N\rightarrow \infty} r_N(z)=0\Leftrightarrow
v\notin L^2(H, \R_+).\] Thus for a given $z\in\C^+$ as $N\rightarrow \infty$ the circles $C_N(z)$ converges either to a circle $C_{\infty}(z)$ or to a point $m_{\infty}(z).$ If $C_N(z)$ converges to a circle, then its radius $r_{\infty}= \lim r_{N}$ is positive and  \ref{radi}  implies  that $ v\in L^2(H,\R_+)$. If $ \tilde{m}_{\infty}$ is any point on $ C_{\infty}(z)$ then $ \tilde{m}_{\infty}$ is inside any $C_N(z)$ for $ N>0.$ Hence \[ \displaystyle \int_0^N (u+\tilde{m}_{\infty}v)^*H(u+\tilde{m}_{\infty}v) <  \frac{ \text{ Im }\tilde{ m}_{\infty}}{ \text{ Im }z}\] and letting $ N \rightarrow \infty$ one sees that $ f(x,z)= u+ \tilde{m}_{\infty}v \in L^2(H,\R_+).$ The same argument holds if $\tilde{m}_{\infty}$ reduces to a point $m_{\infty} $. Therefore, if $ \text {Im }z\neq 0$, there always exists a solution of \ref{ca} of class $\in L^2(H,\R_+)$.In the case $C_N(z)\rightarrow C_{\infty}(z) $ all solutions are in $L^2(H,  \R_+)$ for Im $z\neq 0$ and this identifies the limit-circle case with the existence of the circle  $C_{\infty}(z)$ . Correspondingly, the limit-point case is identified with the existence of the point $ m_{\infty}(z).$ In this case $C_N(z)\rightarrow m_{\infty} $ there results $ \lim r_N =0$ and  \ref{radi}  implies that $v$ is not of class $L^2(H,  \R_+)$.
Therefore in this situation there is only one linearly independent solution of class $L^2(H , \R_+)$. In the limit circle case $m\in C_N$ if and only \ref{mh} holds. Since $f(x,z) = u(x,z) + mv(x,z)$, it follows that $ m$ is on $C_{\infty}$ if and only if \begin{align} \label{cl5} \int_0^{\infty}f(x,z)^*Hf(x,z)dx = \frac{\text {Im} m(z)}{\text{ Im} z}.\end{align} From \ref{gi}, it follows that $m$ is on the limit circle if and only if $\displaystyle \lim_{N\rightarrow \infty} W_N(\bar{f},f)=0 $. From the above discussion we proved

 \begin{theorem}\cite{coddington}  \begin{enumerate}
  \item  The  limit-point case $(r_{\infty}=0 )$
implies that \ref{ca} has precisely one $ L^2 (H, R_+)$
solution.
\item The  limit-circle case $ (r_{\infty}>0 )$ implies all solutions of
 \ref{ca} are in $ L^2 (H, \R_+).$
\end{enumerate}   \end{theorem}
The identity \ref{mh} shows that $ m_{N}^{\beta}(z)$ are holomorphic functions mapping upper-half plane to itself. The poles and zeros of these functions lie on the real line and are simple. In the limit-point case, the  limit  $m_{\infty}(z)$ is a holomorphic function mapping upper-half plane to itself. In limit-circle case, each circle $C_N(z)$ is traced by points $ m = m_N^{\beta}(z)$ as $ \beta $ ranges over $0 \leq \beta < \pi $ for fixed $ N $ and $z$. Let $z_0$ be fixed, $ \text{ Im }z_0 > 0$. A point $ \tilde{m}_{\infty}(z_0)$ on the circle $C_{\infty}(z_0)$ is the limit point of a sequence
$  m_{N_j}^{\beta_j}(z)$ with $ N_j\rightarrow \infty $ as  $ j\rightarrow \infty $.

 In \cite{KA} we showed that $ \operatorname{tr} H \equiv 1$ implies the limit-point case. We would like to consider \ref{ca} to have limit-point case. If the Hamiltonian $ H$ in a canonical system does not have trace norm $1$, then we can always pass the equivalent canonical system with the Hamiltonian $H$ having trace norm $1$. More precisely we  use the change of variable \begin{align} \label{cv}t(x)= \int _0^x \operatorname{tr} H(s) ds .\end{align}
Let $x(t)$ be the inverse function  and define the new Hamiltonian
$ \widetilde{H}(t)= \frac{1}{\operatorname{tr} H (x)} H(x(t))$ so that $ \operatorname{tr} \widetilde {H}(t) \equiv 1$. Let $u(x,z)$ be the solution of the original system \[ Ju'= zHu\] and put $ \widetilde{u}(t,z) = u(x(t),z).$ Then  $ \widetilde{u}(t,z)$ solves the new equation\[J\widetilde{u}'= z \widetilde{H}\widetilde{u} .\] Their corresponding Weyl-m functions on a compact interval $[0,N]$ are the same up to the change of the point of boundary condition, ie $\tilde{m}_N^\beta(z) = m_{x(N)}^\beta(z)$. From now onward we will consider a canonical system with $\operatorname{tr} H \equiv 1.$ This will reduce our system to have limit-point case.

\section{Topologies on the space of Hamiltonian.} \label{topo}

We need to consider the space of Hamiltonians  and a suitable topology on it so that the space is compact. With the topology we have, we want  to work with the basic object like $m$ functions of the canonical system. Let $M(\R)$ denotes the set of Borel measures on $\R.$ Consider the space \begin{align*}
\mathcal V_{2\times 2} = \big\{\mu \in M(\R)^{2\times2}: d \mu =
H(x)dx, H(x)\geq 0, \operatorname{tr}H(x)\equiv 1, H(x)\in
L^1_{\text{loc}} \big \}.\end{align*}
 We would like to define a metric on  $\mathcal V_{2\times 2} $. We will follow the same procedure as in \ref{RT}. Let $ C_{c}(\R)$ denotes the space of all continuous functions on $\R$ with compact support, the continuous functions vanishing outside of a bounded interval. This space $ C_{c}(\R)$ is complete with respect to the $\|.\|_{\infty}$ norm. Pick a
countable dense subset $\{f_n: \,\, n\in\N \} \subset C_{c}(\R),$
 the continuous functions of compact support. Let
 \begin{align*} \rho_n(\mu, \nu)= \sum_{1\leq i, j\leq 2} \Big|\int f_n
 d(\mu_{ij}-\nu_{ij})(x)\Big|.\end{align*} Then define a metric $d$  on  $\mathcal V_{2\times2}.$ as \begin{align*} d(\mu, \nu)=
 \sum_{n=1}^{\infty}2^{-n}\frac{ \rho_n(\mu, \nu)}{ 1+ \rho_n(\mu,
 \nu)}.\end{align*} Clearly $d$ is a metric on $\mathcal V_{2\times2}.$ Moreover, $\Big(\mathcal V_{2\times2},d\Big)$ is a compact metric space.

 We can now consider canonical system with measures as Hamiltonian.
 \begin{align}\label{ca5} Ju'=z\mu u ,\,\, \mu  \in \mathcal V_{2\times2}.\end{align} If $I\subset\R$
 is an compact interval and $ B(I)$ denotes the space of all  complex valued bounded
 Borel measurable functions on $I$.Then $B(I)$ is complete with respect to the metric given by $\rho(f,g)=||f-g||_u $ where the norm on $B(I)$ is $||f||_u= \text{sup}_{x\in I}|f(x)|$.
   Let $B(I)^2= \Big\{f= \begin{pmatrix}f_1\\f_2 \end{pmatrix} \,\,\, f_1,f_2 \in
 B(I)\Big\}.$ Clearly the space $B(I)^2$ is complete with respect
 to the metric given by $\rho(f,g)=||f-g||_u $ where $||f||_u =\text{sup}_{x\in I}|f_1(x)|+ \text{sup}_{x\in I}|f_2(x)|$.
   Let $f\in B(I)^2,$ we call $f$ a solution to
 the equation \ref{ca5} if and only if \[ J(u(x)-u(a+)) = z \int_{(a,x)}d \mu(t)u(t)\,\, \text{ if } x\geq a \geq 0
   \] and \[J(u(x)-u(a-)) = - z\int_{(x,a)}d \mu(t)u(t)\,\, \text{ if } x\leq a \leq 0.\]
   In order to show the existence of a solution  of the equation
   \ref{ca5}, define a map  on $B(I)^2$ by \begin{align*} Tu(x) = u(0) - z J\int_0^x
\mu(t)u(t)\end{align*} and show that $T$ is a contraction mapping.
   \begin{align*}\|Tu-Tv \|_u &= \text{sup}_{x\in I}\Big|z \int_0^x
(u_1-v_1)d\mu_{11} +(u_2-v_2)d\mu_{12} \Big| \\&+ \text{sup}_{x\in
I}\Big|z \int_0^x (u_1-v_1)d\mu_{21} +(u_2-v_2)d\mu_{22} \Big|
\\ & \leq \text{sup}_{x\in I}\Big[|z|\int_0^x |u_1-v_1|d\mu_{11}
+|z||u_2-v_2|d\mu_{12}\\&+\text{sup}_{x\in I}|z| \int_0^x
|u_1-v_1|d\mu_{21} +|z||u_2-v_2|d\mu_{22}\Big]\\ & \leq
\frac{c}{2} \text{sup}_{x\in I}\Big[\text{sup}_{t\in
[0,x]}|u_1-v_1|+\text{sup}_{t\in [0,x]}|u_2-v_2|\Big]\\&+
\frac{c}{2}\text{sup}_{t\in [0,x]}\Big[\text{sup}_{t\in
[0,x]}|u_1-v_1|+ \text{sup}_{t\in [0,x]}|u_2-v_2|\Big]\\ & \leq
\frac{c}{2}\Big[2 \text{sup}_{x\in I}|u_1-v_1|+2\text{sup}_{x\in
I}|u_2-v_2|\Big]\\&=c\|u-v\|_u.
\end{align*}

 This shows that $T$ is a contraction mapping, so it has a unique
fixed point say $u(x)$ in $B(I)^2$ such that $Tu(x) = u(x).$ So
there is a solution in $B(I)^2$ that satisfy $ u(x)= u(0) - z J\int_0^x \mu(t)u(t).$ As
already seen that $\operatorname{tr} H(x)\equiv 1 $ implies the limit
point at the both end points.

  This means that for $z\in \C^{+}$
there exists (unique up to a factor) solutions $f_{\pm}(x,z) = u(x,z) \pm m_{\pm} (z) v(x,z) $ of
\ref{ca5} such that $f_{-} \in L^2(H, \R_-), f_{+}\in L^2(H,
\R_+)$ where $ u(x,z)$ and $v(x,z)$ are any two linearly independent solutions of \ref{ca5}. Let $ x\in \R$, and $\alpha = 0$, the Dirichlet boundary
condition at $x$ that is  $ u_1(x,z)= v_2(x,z)= 0,  v_1(x,z)= u_2(x,z) = 1$,   the Titchmarsh-Weyl $m$-functions of the system \ref{ca5} are alternately defined as $m_{\pm} (x,z)= \pm
\frac{f_{\pm_2}(x,z)}{f_{\pm_1}(x,z)}.$ Recall that $m_{\pm}(x,z)$ are Herglotz functions. So
the boundary value of these $m$ functions  are defined by
$m_{\pm}(x,t)\equiv \lim_{y\rightarrow 0}m_{\pm}(x,t+iy).$
\
\begin{defi} Let $A\subset\R$ be a Borel set. We call a Hamiltonian $\mu \in  \mathcal V_{2\times 2}$  reflectionless on $A$ if
\begin{align}\label{ref}m_{+}(x,t)=-\overline{m_{-}(x,t)} \end{align} for almost every  $ t\in A$ and for some $x\in \R$.
\end{defi}
The set of reflectionless hamiltonian  on $A$ is denoted by $\mathcal R(A).$ Notice that the equation
\ref{ref} is independent of the choice of boundary condition and
the choice of a point. Suppose \ref{ref} is true for a boundary
condition $\alpha$ at $0$. Let $m_+^{\alpha}(z)$ be the unique coefficient such that $f(x,z)=u_{\alpha}(x,z)+m_+^{\alpha}(z)v_{\alpha}(x,z) \in L^2(H,\R_+).$  Suppose $ T_{\alpha}(x,z)=
    \begin{pmatrix}u_{\alpha _1}(x,z)&v_{\alpha _1}(x,z)\\u_{\alpha_2}(x,z)&v_{\alpha_2}(x,z)\end{pmatrix}$ with  $ T_{\alpha}(0,z)=
    \begin{pmatrix}\sin{\alpha}&\cos{\alpha}\\-\cos{\alpha}&\sin{\alpha}\end{pmatrix}$ and  $ T_{\beta}(x,z)=\begin{pmatrix}u_{\beta _1}(x,z)&v_{\beta _1}(x,z)\\u_{\beta_2}(x,z)&v_{\beta_2}(x,z)\end{pmatrix}$ with  $ T_{\beta}(0,z)=
    \begin{pmatrix}\sin{\beta}&\cos{\beta}\\-\cos{\beta}&\sin{\beta}\end{pmatrix} .$  Then \\ $
    T_{\alpha}(x,z)=T_{\beta}(x,z)\begin{pmatrix}\cos{\gamma}&\sin{\gamma}\\-\sin{\gamma}&\cos{\gamma}\end{pmatrix}$,
      where $\gamma= \beta-\alpha.$ \\

     Here $m_+^{\alpha}(z)\in\C$ is a
    unique number such that  \begin{align*} & T_{\alpha}(x,z)\begin{pmatrix}1\\ m_+^{\alpha}(z) \end{pmatrix} \in
    L^2(H,\R_+).\\ \Rightarrow & T_{\beta}(x,z)\begin{pmatrix}\cos{\gamma}&\sin{\gamma}
    \\-\sin{\gamma}& \cos{\gamma}\end{pmatrix} \begin{pmatrix}1 \\m_+^{\alpha}(z)\end{pmatrix} \in
    L^2(H,\R_+).\\ \Rightarrow & T_{\beta}(x,z)\begin{pmatrix}\cos{\gamma}+m_+^{\alpha}(z)\sin{\gamma}
    \\-\sin{\gamma}+m_+^{\alpha}(z)\cos{\gamma}\end{pmatrix} \in L^2(H,\R_+).\\ \Rightarrow & (\cos{\gamma}+m_+^{\alpha}(z)\sin{\gamma})T_{\beta}(x,z)\begin{pmatrix}1
    \\
    \frac{-\sin{\gamma}+m_+^{\alpha}(z)\cos{\gamma}}{\cos{\gamma}+m_+^{\alpha}(z)\sin{\gamma}}\end{pmatrix}\in
    L^2(H,\R_+).\end{align*} Since $  m_+^{\beta}(z)$ be the unique coefficient such that $ T_{\beta}(x,z)\begin{pmatrix}1\\ m_+^{\beta}(z) \end{pmatrix} \in
    L^2(H,\R_+) $  we must have,\begin{align*}
 & m_+^{\beta}(z)
=\frac{-\sin{\gamma}+m_+^{\alpha}(z)\cos{\gamma}}{\cos{\gamma}+m_+^{\alpha}(z)\sin{\gamma}} = \begin{pmatrix}\cos{\gamma}&-\sin{\gamma}\\\sin{\gamma}&\cos{\gamma}\end{pmatrix}m_+^{\alpha}(z).\end{align*}
On the other hand, exactly in the same way,
\begin{align*} & T_{\alpha}(x,z) \begin{pmatrix}1\\ -m_-^{\alpha}(z)
\end{pmatrix}\in L^2(H,\R_-).  \\\Rightarrow  & T_{\beta}(x,z)\begin{pmatrix}\cos{\gamma}&\sin{\gamma}\\-\sin{\gamma}&\cos{\gamma}\end{pmatrix}
    \begin{pmatrix}1\\ -m_-^{\alpha}(z) \end{pmatrix}\in L^2(H,\R_-),    \text{ where }\gamma=
    \beta-\alpha.\\      \Rightarrow  & T_{\beta}(x,z)\begin{pmatrix}\cos{\gamma}-
m_-^{\alpha}(z)\sin{\gamma}\\-\sin{\gamma}-m_-^{\alpha}(z)\cos{\gamma}\end{pmatrix}
   \in L^2(H,(-\infty,0])\\   \Rightarrow & (\cos{\gamma}- m_-^{\alpha}(z)\sin{\gamma})T_{\beta}(x,z)
   \begin{pmatrix}1\\ \frac{-\sin{\gamma}-m_-^{\alpha}(z)\cos{\gamma}}{\cos{\gamma}- m_-^{\alpha}(z)\sin{\gamma}}\end{pmatrix}\in L^2(H,\R_-)
\\  \Rightarrow & -m_-^{\beta}(z) =\frac{-\sin{\gamma}-m_-^{\alpha}(z)\cos{\gamma}}{\cos{\gamma}-
m_-^{\alpha}(z)\sin{\gamma}}
 =  \begin{pmatrix}\cos{\gamma}&\sin{\gamma}\\-\sin{\gamma}&\cos{\gamma}\end{pmatrix}m_-^{\alpha}(z).\end{align*}
Let $
P_+(0,z)=\begin{pmatrix}\cos{\gamma}&-\sin{\gamma}\\\sin{\gamma}&\cos{\gamma}\end{pmatrix}\text{
and }
P_-(0,z)=\begin{pmatrix}\cos{\gamma}&-\sin{\gamma}\\\sin{\gamma}&\cos{\gamma}\end{pmatrix},
$ so that \\ $ m_-^{\beta}(z)=
P_-(0,z)m_-^{\alpha}(z),\,\, m_+^{\beta}(z)=
P_+(0,z)m_+^{\alpha}(z) \text{ and } \\
\begin{pmatrix}1& 0\\ 0& -1\end{pmatrix}P_+(0,z)= P_-(0,z)\begin{pmatrix}1& 0\\ 0&
-1\end{pmatrix}.$ By simple calculation we can see that \[
m_+^{\beta}(t)=-\overline{m_-^{\beta}(t)}. \] Similarly,
equation \ref{ref} is independent of the choice of the point.
Suppose \\$T_0(x,z)= \begin{pmatrix}u_1(x,z)& v_1(x,z)\\
u_2(x,z)& v_2(x,z)\end{pmatrix}$ be  solutions with the boundary
conditions at $0$. Then  $T_0(x,z)= T_a(x,z)\begin{pmatrix}u_1(a,z)& v_1(a,z)\\
u_2(a,z)& v_2(a,z)\end{pmatrix}.$ Suppose $m_{\pm}(0,z) \in \C $
be the unique coefficients such that $ f_{\pm}(x,z)= u(x,z)\pm
m_{\pm}(0,z)v(x,z) \in L^2(H,\R_{\pm}).$ \\ In another way,
$ T_0(x,z)\begin{pmatrix}1 \\
\pm m_{\pm}(0,z) \end{pmatrix}\in L^2(H,\R_{\pm}) . $ \\$  \Rightarrow
  T_a(x,z)\begin{pmatrix}u_1(a,z)& v_1(a,z)\\
u_2(a,z)& v_2(a,z)\end{pmatrix}\begin{pmatrix}1\\
\pm m_{\pm}(0,z) \end{pmatrix}\in L^2(H,\R_{\pm})
 .\\  \Rightarrow   m_{\pm}(a,z)  = \frac{u_2(a,z)\pm m_{\pm}(0,z)v_2(a,z)}{u_1(a,z)\pm m_{\pm}(0,z)v_1(a,z)}  =
\begin{pmatrix}v_2(a,z)&  \pm u_2(a,z)\\
\pm v_1(a,z)& u_1(a,z)\end{pmatrix}m_{\pm}(0,z). $ \\ Let
$T_{\pm}(z)=\begin{pmatrix}v_2(a,z)&  \pm u_2(a,z)\\
\pm v_1(a,z)& u_1(a,z)\end{pmatrix} $, then $\begin{pmatrix}1& 0\\
0& -1\end{pmatrix} T_+(z)= T_-(z)\begin{pmatrix}1&
0\\0&-1\end{pmatrix}.$ By calculation we see that\[  m_+(a,t) = - \overline{m_-(a,t)} \]

As already mentioned, the Weyl m-functions $m(x,.)$ are Herglotz
functions and by the Herglotz representation theorem  they have unique integral representation of the
form,
\begin{align*} m(x,z)= a + bz + \int_{\R}\Big(\ \frac{1}{t-z}-\frac{t}{t^2+1}\Big)d\nu(t), \,\,\,\,\, z\in \C^+ \end{align*} for some positive Borel measure $\nu$
on $\R$ with $\int\frac{1}{t^2+1}d\nu < \infty$ and numbers $a\in
\R,\, b\geq 0.$  We call the measure $\nu$ in above integral
representation of $m$ as  { \em spectral measure} of the system  \ref{ca}.

Recall that a Borel measure $\rho$ on $\R$ is called {\em absolutely continuous } if  $ \rho(B)=0$ for all Borel sets $B\subset\R$ of Lebesgue measure zero. By the Radon-Nikodym Theorem, $\rho$ is absolutely continuous if and only if $d\rho= f(t)dt$ for some density $f \in L_{loc}^1(\R), \,\, f\geq 0.$ If $\nu$ is supported by a Lebesgue null set that is, there exists a Borel set $B\subset\R$ with $ |B|=\rho(B^c)=0$, then we say that $\rho$ is {\em singular}. By Lebesgue's decomposition theorem, any Borel measure $\rho$ on $\R$ can uniquely decomposed into absolutely continuous and singular parts:\[ \rho = \rho_{ac} + \rho_{s}. \] The  { \em essential support } $ \Sigma _e $ of a Borel measure $\rho$ on $\R$ is the complement of a largest open set $ U\subset \R $ such that $ \rho(U)=0.$

Let  $\mu$ be a measure on $\R.$
The shift by $x$ of the measure $\mu$ , denoted by $S_x\mu$, is  defined by \[ \int_{\R}
f(t)d(S_x\mu) = \int_{\R} f(t-x)d\mu(t).\] If $\mu \in \mathcal V_{2\times 2}$ is such that $d\mu= H(t)dt $ is a
locally integrable Hamiltonian $ H $, then this reduces to the shift map
$(S_xH)(t)= H(x+t).$ \begin{defi} The $ \omega $ limit set of the Hamiltonian
$\mu \in \mathcal V_{2\times 2}$ under the shift map is defined as,
\[ \omega(\mu)= \{ \nu \in \mathcal V_{2\times 2}: \text { there exist $x_n\rightarrow \infty $ so that } d(S_{x_n}\mu,\nu)\rightarrow 0\}.\]
 \end{defi} Then as in \cite{CR1} we can see that $ \displaystyle \omega(\mu) \subset \mathcal V_{2\times 2}$  is compact, non-empty and $S$ is a homeomorphism  on $ \omega(\mu).$ Moreover, $ \omega(\mu)$ is connected.

\section{Main theorem and its proof } \label{mainth}

 We are now ready to
state the Remling's theorem for Canonical System on $\R_+.$
\begin{theorem}
 \label{RT} Let $\mu \in \mathcal V_{2\times 2} $  be a (half line)
Hamiltonian, and let $\Sigma_{ac}$ be the essential support of the
absolutely continuously part of the spectral measure. Then \[
\omega(\mu)\subset \mathcal R ( \Sigma_{ac}).\] \end{theorem} In
order to prove this theorem we approach the similar way as in
\cite{CR1}.\\
Let $\mu \in \mathcal V_{2\times 2}$ is a whole line Hamiltonian. We
write $\mu_{\pm}$ for the restrictions of $\mu$ to $\R_{\pm}.$
Denote the set of restrictions by $\mathcal V_{\pm} = \{\mu_{\pm} : \mu
\in \mathcal V_{2\times 2} \}.$\\

Let $\mathbb{H}$ denote the set of all Herglotz functions, that is $ \mathbb{H} = \{ F: \C^+ \rightarrow
\C^+ : F \text{ is holomorphic } \}.$ So $ \{M_{\pm}=
m^{\mu}_{\pm}(0,z) \}\subset \mathbb H .$  First lets prove the following lemma.

\begin{lemma}\label{Lem6.2} The maps $ \mathcal V_{\pm} \longmapsto \mathbb H , \,  \mu_{\pm}\longmapsto M_{\pm}= m^{\mu}_{\pm}(0,z)$
 are homeomorphism onto their images.\end{lemma}

\begin{proof}  We have $\mu_+ = H_+(x)dx$. By Theorem 1 in \cite{HW1} for every canonical system with Hamiltonian  $H_+$ there is unique
$m_+(0,z).$ Conversely for every $m_+ \in \mathbb H$ there exists
a unique Hamiltonian $H_+$ on $\R_+$ such that $m_+$ is a Wyle
coefficient of the canonical system corresponding to $H_+$. So
$\mu_+ \mapsto M_+$ is one-to-one. Next we show that the map is
homeomorphism. Consider the canonical system \ref{ca} on $\R_+.$
Suppose $\mu_n\rightarrow \mu $ in $\nu^C_+$. That is $H_n(x)dx
\rightarrow H(x)dx$ for some Hamiltonian $ H_n(x),  H(x) .$ Let
$u_n$ be the solution of Canonical System with Hamiltonian $
H_n(x).$  Let $K$ be a compact subset of $\C^+$ contained in a
ball $B(0,R).$ Suppose a subinterval $[0,\eta]$ be such that $\eta
= \frac{1}{8R}.$ We claim that $u_n$ has convergent subsequence on
 $[0,\eta].$ Define the operators $T_n: C[0,\eta]\longrightarrow C[0,\eta]
 $ by \[ T_n u(x) = -zJ\int _0^x H_n(t)u(t)dt. \] Since \begin{align*}\|T_n\| = & \sup_{\|u\|_{\infty}=1}\|-zJ\int _0^x
 H_n(t)u(t)dt\|\\ \leq & |z|\|u\|_{\infty}\displaystyle \int _0^x |H_n(t)|dt \\ & \leq R4\eta = R4\frac{1}{8R}=
 \frac{1}{2},\end{align*} $\|T_n\| $ are uniformly bounded.
 So the Neumann series $(1-T_n)^{-1}= \displaystyle \sum _{k=0}^{\infty}
 T_n^k$ is convergent. Here $ u_n(x) = (1-T_n)^{-1}(u_0),\,\,\,\ u_0= \begin{pmatrix}1\\
 0\end{pmatrix}.$\\ $\|u_n\|\leq \|(1-T_n)^{-1}\| \|u_0\| = \|(1-T_n)^{-1}\| \leq \displaystyle \sum _{k=0}^{\infty}
 \|T_n\|^k = \displaystyle \sum _{k=0}^{\infty}(\frac{1}{2})^k = 2
 .$ So $\{u_n(x)= n\in \N \}$ is uniformly bounded in $n$ on
 $[0,\eta]$ and locally uniformly in $z$. Similar argument shows
 that $u_n$ remains bounded on $[\eta, \eta+p]$ so that $u_n$ are
 eventually bounded uniformly on $[0,N].$ Moreover, $u_n$ are
 equicontinuous. Let $\epsilon > 0$ be given. Since $u_n$ are
 solutions for the system \ref{ca} we have, \[u_n(x)-u_n(x_0)= -zJ\int _{x_0}^x
 H_n(t)u_n(t)dt.\]\begin{align*} \|u_n(x)-u_n(x_0)\|  & \leq  |z|\|u_n\|\int _{x_0}^x |H_n(t)|dt
  \\ &= |z|\|u_n\| 4\eta\|x-x_0\| \\ &\leq R
  2.4\eta\|x-x_0\|.\end{align*} Let  $\delta =
  \frac{\epsilon}{8R\eta}$  then  $ \|u_n(x)-u_n(x_0)\| < \epsilon, \text{  if  }  \|x-x_0\| < \delta  \text{   for all n
  }.$ By Arzella-Ascolli Theorem  $\{u_n \}$ has convergent
  subsequence say $u_{n_j}\rightarrow u.$ We show that $u$
  satisfies the Canonical System corresponding to $H(x).$ \begin{align*}u_{n_j}(x)-u_{n_j}(0)  & = -zJ\int _{0}^x
 H_{n_j}(t)u_{n_j}(t)dt \\ & = -zJ\int _{0}^x H_{n_j}\big((t)u_{n_j}(t)- u(t) \big)dt -zJ\int _{0}^x
 H_{n_j}(t)u(t)dt . \end{align*} Since  $\parallel -zJ\int _{0}^x H_{n_j}\big((t)u_{n_j}(t)- u(t) \big)dt
 \parallel\leq |z|\| H_{n_j}\|_{L_1(0,x)} \|u_{n_j}- u \|$ ,\\
  $\displaystyle \lim_{j\rightarrow \infty} -zJ\int _{0}^x H_{n_j}\big((t)u_{n_j}(t)- u(t) \big)dt =
  0.$ Hence, taking the limit as $j\rightarrow \infty$ we get,
  $ u(x)-u(0) = \int _{0}^x H(t)u(t)dt . $ So
   $\mu_n\rightarrow \mu  \Rightarrow  u_n \rightarrow u \Rightarrow m_+^{\mu_n}(0,z)\rightarrow
   m_+(0,z).$ This proves the continuity of the map on the
   interval $[0,N].$  Inverse of a continuous map on compact set
   is also continuous. Hence the map is homeomorphic. Exactly, the
   same way $\mu_-\longleftrightarrow M_-$ is also a
   homeomorphism.
\end{proof}

\subsection{ Breimesser-Pearson Theorem on canonical systems }

 For $z=x+iy \in \C^+, \,\, \omega_{z}(S) = \frac{1}{\pi}\int_{S}\frac{y}{(t-x)^2+y^2}dt$,   denotes the harmonic measure in the upper-half plane.
 For any $G \in \mathbb H $ and $t \in \R $ we define $ \omega_{ G(t)}(S)$ as the limit  \[ \omega_{ G(t)}(S) = \displaystyle \lim_{y\rightarrow 0+}\omega_{ G(t+iy)}(S).\] For complete description about the Herglotz functions and harmonic measures, see \cite{CR1}.

\begin{lemma}\cite{CR1} \label{lem1.4}  Let $A\subset\R$ be a Borel Set with $|A|<\infty$. Then \begin{align*} \displaystyle \lim_{y\rightarrow 0+}  \sup_ {F \in \mathcal H ; S \subset\R}  \Big| \int_A \omega_{ F(t+iy)}(S)dt - \int_A \omega_{F(t)}(S)dt \Big|= 0   \end{align*}. \end{lemma}

\begin{defi} If $F_n, F \in \mathbb H ,$ we say that $F_n \rightarrow F$ in value distribution if
  \begin{align}\label{def1} \displaystyle \lim_{n\rightarrow \infty} \int_A \omega_{F_n(t)}(S)dt =\int_A \omega_{F(t)}(S)dt
\end{align}for all Borel set $A, S \subset \R, |A|<\infty.$
 \end{defi} Notice that if the limit in the value distribution exists, it is unique.
\begin{theorem} \label{VD}\cite{CR1} Suppose $ F_n, F \in \mathcal H ,$ and let $a_n, a, $ and $ \nu_n, \nu$ be the associated numbers and measures, respectively, from the integral representation of Herglotz function. Then the following are equivalent:\begin{enumerate}
\item $F_n(z)\rightarrow F(z)$ uniformly on compact subsets of $\C^+;$
 \item $a_n \rightarrow a$ and $\nu_n \rightarrow \nu$ weak * on $ \mathcal M (\R_{\infty}),$ that is,\[ \displaystyle \lim_{n\rightarrow \infty} \int_{\R_{\infty}}f(t)d\nu_n(t)= \int_{\R_{\infty}}f(t)d\nu(t)\] for all $f\in C(\R_{\infty})$;
   \item  $ F_n\rightarrow F $ in value distribution.
  \end{enumerate}
\end{theorem}
Remling's theorem is in fact a reformulation of  Breimesser-Pearson theorem which state as follows

 \begin{theorem} \label{Bp}  Consider a half-line
Canonical System. Let $\Sigma_{ac}$ denotes the essential support
of absolutely continuous part of Spectral measure then for any
$A\subset\Sigma_{ac}, |A|< \infty $ and $S\subset\R$, we have
\[\displaystyle \lim_{N\rightarrow \infty} \Big( \int_{A}\omega_{m_{-}(N,t)}(-S)dt - \int_{A}\omega_{m_{+}(N,t)}(S)dt\Big)=
0.\]Moreover, the convergence is uniform in $S$.

\end{theorem}

We  prove this theorem on canonical systems using the same technique as in \cite{CR1}.

The hyperbolic
distance of two points $w,z \in \C^+$ is defined as \[ \gamma(w,z)
= \frac{|w-z|}{\sqrt{ \text{Im}w}\sqrt{\text{Im}z}}.\] Hyperbolic distance and harmonic measure are intimately related as follows,
\[|\omega_{w}(S)-\omega_{z}(S)|\leq \gamma(w,z)\] for any $z,w \in \C^+ $ and any Borel set $S\subset\R.$ Moreover, if $F(z)= \alpha(z)+i\omega_{z}(S)$ , $\alpha(z)$ is a harmonic conjugate of $\omega_{z}(S)$ we have

 \begin{align} \label{hyprdistance}|\omega_{w}(S)-\omega_{z}(S)|\leq \frac{|\omega_{w}(S)-\omega_{z}(S)|}{
\sqrt{\omega_{w}(S)} \sqrt{\omega_{z}(S)}} \leq \gamma (F(w), F(z))\leq
\gamma(w,z).\end{align}

\begin{lemma}  Let $u(.,z), v(.,z)$ be the solution of the Canonical
System \ref{ca}, subject to the condition $u(0,z)=
\begin{pmatrix}1\\0\end{pmatrix}, v(0,z)= \begin{pmatrix}0\\1\end{pmatrix}.$
Let $w$ be any constant such that $ \text{Im}w \geq 0, $ for any $
N>0,$ and all $ z \in \C^+$, we have the estimate, \[ \gamma \Big(
-\frac{v_2(N,z)}{v_1(N,z)}, - \frac{u_2(N,z)+
\bar{w}v_2(N,z)}{u_1(N,z)+ \bar{w}v_1(N,z)} \Big) \leq
\frac{1}{\sqrt {I(I+1)}},\] where  $I = I(N,z)$ is the integral
defined by $ I(N,z)= ( \text{ Im}z) \int_0^N Im( u^* H v)dx.$
\end{lemma}

\begin{proof} Denote the wronskian $ W_N(f,g)= f_1(N)g_2(N)-f_2(N)g_1(N)).$
 Using the Greens's Identity we have, \label{lem1.1}\begin{align} \int_0^N v^*Hv dx = \frac{1}{2i \text{ Im}z} W_N(v,\bar{v}) , \end{align}

 \label{lem1.2}\begin{equation}\int_0^N \text{ Im}(u^*Hv) dx  = - \frac{1}{2 \text{ Im}z} \Big( 1 - \text { Re}W_N(\bar{u},v) \Big)
= \frac{1}{2 \text{ Im}z}\Big( 1 - \text { Re}W_N(u,\bar{v})
\Big),\end{equation}
 \label{lem1.3}\begin{align}|W(u,\bar{v})|^2= 1-  W(u,\bar{u})W(v,\bar{v}).\end{align} Now at $x=N$, we have,
 \begin{align*}\gamma^2\Big( -\frac{v_2}{v_1}, - \frac{u_2+\bar{w} v_2}{u_1+\bar{w} v_1}\Big)
     =  -\frac{4}{ W (v, \bar{v}) W (u+\bar{w} v,\bar{u}+w\bar{v}  )}.\end{align*} Therefore,
     \[\gamma^2\Big( -\frac{v_2}{v_1}, - \frac{u_2+\bar{w} v_2}{u_1+\bar{w} v_1}\Big) \leq -\frac{4}{ W (v, \bar{v}) W (u+\bar{w} v,\bar{u}+w\bar{v}  )}.\]
       Let $w$ be real. The denominator on the right side is of the form $A+Bw+Cw^2,$ where $A\geq 0, C\geq 0$ and $B$ is real.
         The denominator has minimum value $A-\frac{B^2}{4C}.$ Hence, \begin{align*} \gamma^2 &  \leq    \frac{4}{-W (v, \bar{v})W (u, \bar{u})
         -\frac{\big(  W (v, \bar{v}) ( W (u, \bar{v}) - W (\bar{u},v))\big)^2 }{ 4 (- W (v, \bar{v})^2 )}} \\  & \leq   \frac{-4}{W (v, \bar{v}) ( W (u, \bar{v})+ \text{ Im} ( W (u, \bar{v})\big)^2 } . \end{align*} Using equation \ref{lem1.3} we get,
         \begin{align*} \gamma ^2 & \leq   - \frac{4}{ 1- |W (u, \bar{v}) |^2+ (\text{ Im} ( W (u, \bar{v})\big)^2 )} \\  & =   \frac{-4}{1-(\text{Re}W (u, \bar{v}))^2}
         .\end{align*} Here,  \begin{align*} 1-(\text{Re}W (u, \bar{v}))^2  & = (1-(\text{Re}W (u, \bar{v})))(1+(\text{Re}W (u, \bar{v})))
         \\ & = \Big( -2 \text{ Im} z \int_0^N \text{ Im}(u^*H v)dx\Big) \Big( 1+ 2 \text{ Im} z \int_0^N \text{ Im}(u^*H v)dx\Big)
         .\end{align*} Therefore,\begin{align*} \gamma ^2  \leq  \frac{1}{I(1+I)} \text{   where }  I =  \text{ Im} z \int_0^N \text{ Im}(u^*H v)dx. \end{align*}
         If $w$ is not real, $w = \text{ Re}w +iY , Y>0 $ then $ u -iYv $ is also a solution and we have,\begin{align*}
         | W(u-iYv, \bar{v})|^2 = 1-W(u-iYv, \bar{u}+ iY\bar{v})W (v, \bar{v}) .\end{align*} Also from above equation,
         \begin{align*} \gamma ^2 &   \leq    \frac{- 4}{ W (v, \bar{v})W (u, \bar{u})+ (\text{ Im} ( W (u, \bar{v})\big)^2 )   + Y^2 W (v, \bar{v})^2
         +2i Y\text{Re}W (u, \bar{v})W (v, \bar{v}) }  \\   &    \leq  \frac{-4}{W(u-iYv, \bar{u}+ iY\bar{v})W (v, \bar{v})+\Big(
\text{ Im}W(u-iYv, \bar{v}) \Big)^2}
         .\end{align*} Since the equation \ref{lem1.3} is valid for $ u-iYv$ we get,
         \begin{align*} \gamma ^2   & \leq  \frac{-4}{1- \big( \text { Re}W(u-iYv, \bar{v}) \big)^2} \\
          & =  \frac{-4}{\big( 1+ \text { Re}W(u-iYv, \bar{v}) \big) \big(1- \text { Re}W(u-iYv, \bar{v}) \big)} \\
           & = \frac{-4}{\Big( 1+ \text { Re}\big(W(u ,\bar{v})-iYW (v, \bar{v})\big) \Big) \Big(1- \text { Re}\big(W(u ,\bar{v})- iYW (v, \bar{v})\big)\Big)}\\
          & = \frac{-4}{ \big( 1- \text { Re}W(u ,\bar{v})- Y \text{ Im}W (v, \bar{v})\big) \big( 1+ \text { Re}W(u ,\bar{v})+Y \text{ Im}W (v, \bar{v})\big)}\\
           &= \frac{-4}{ \Big( -2\text{ Im}z \int_0^N \text{ Im}(u^*H v)dx- \frac{Y}{i}W (v, \bar{v})\Big)\Big( 2\text{ Im}z \int_0^N \text{ Im}(u^*H v)dx +2+\frac{Y}{i}W (v, \bar{v})\Big)}\\
           &= \frac{1}{ I'(I'+1)}, \end{align*} where   $ I'= \text{ Im}z \int_0^N \text{ Im}(u^*H v)dx +2+\frac{Y}{2i}W (v, \bar{v}).$
            Notice that $I'\geq I$ since $ W (v, \bar{v}) = 2i\text{ Im}z \int_0^N v^*H vdx \geq 0. $  Hence the lemma is proved for general case.
            \end{proof}
\label{bpc1} \begin{cor} With the  notation  above, we have  \begin{align*} \displaystyle \lim_{N \rightarrow  \infty }\gamma \Big(
-\frac{v_2(N,z)}{v_1(N,z)}, - \frac{u_2(N,z)+
\bar{w}v_2(N,z)}{u_1(N,z)+ \bar{w}v_1(N,z)} \Big) = 0 \end{align*}
\end{cor}
\begin{proof}  From above lemma we have \[ \gamma \Big(
-\frac{v_2(N,z)}{v_1(N,z)}, - \frac{u_2(N,z)+
\bar{w}v_2(N,z)}{u_1(N,z)+ \bar{w}v_1(N,z)} \Big) \leq
\frac{1}{\sqrt {I(I+1)}},\] where  $I = I(N,z)$ is the integral
defined by $ I(N,z)= ( \text{ Im}z) \int_0^N Im( u^* H v)dx.$ Want
to show that $ I \rightarrow \infty $ as $ N\rightarrow \infty.$
We have, \[  \int_0^N v^*H vdx = \frac{1}{2i \text{
Im}z}W_N(v,\bar{v}) \] \[ \int_0^N \text{ Im}(u^*H v)dx = -
\frac{1}{2i \text{ Im}z} \big( 1- \text{ Re}W_N(u,\bar{v})
\big).\] Now lets look at the ratio \begin{align*}  \frac{ 2
\text{ Im}z \int_0^N \text{ Im}( u^* H v)dx +1}{ 2i \text{ Im}z
\int_0^N v^* H vdx}  & = \frac{W_N(u,\bar{v})+W_N(\bar{u},v )}{2i
W_N(v,\bar{v})}
\\ & = \frac{W_N(u,\bar{v})}{ 2i W_N(v,\bar{v})}-
\frac{W_N(\bar{u},v )}{2i W_N(\bar{v},v )} \\ & = \text{ Im}C
\end{align*} where $C$ is the center of the Weyl circle. Since the center of the Weyl circle is continuously depend on $z $ it is uniformly bounded on a compact subset of $ \C^+$. So \[ \int_0^N \text{ Im}( u^* H v)dx +1= \text{ Im}C\int_0^N v^* H vdx\rightarrow\infty \text{ as } N\rightarrow\infty.\]
This implies that $ I \rightarrow \infty $ as $ n\rightarrow
\infty.$\end{proof} We are now ready to prove Theorem \ref{Bp}. We follow the similar approach for the proof of Theorem \ref{Bp} as in \cite{BP1}.

\emph{Proof of Theorem \ref{Bp} : } Let $A\subset \Sigma_{ac},\,\, |A|< \infty$ and let $ \epsilon >0$ be given. We first define a partition
$A= A_0 \cup A_1 \cup A_2, ....\cup A_N$ of disjoint subsets such
that $|A_0|< \epsilon, A_j$ is bounded for $j\geq 1$. We also
require that $m_+(t)\equiv \lim_{y\rightarrow0+}m_+(t+iy)$ exists
and $m_+(t) \in \C^+$ on $\bigcup_{j=1}^N A_j.$  To find $A_j$'s  with these properties, first of all put all $t\in A $ for which $m_+(t)$ does not exist or does not lie in $\C^+$ into $A_0$. Then pick (sufficiently large) compact subset $ K \subset \C^+ , K' \subset \R$ so that $A_0 = \{ t\in A :m_+(t)\notin K \text{ or } t \notin K' \}$ satisfies $|A_0| < \epsilon.$ Subdivide $K$ into finitely many subsets of hyperbolic diameter less than  $\epsilon$, then take the inverse images under $m_+$ of these subsets, and finally intersect with $K'$ to obtain the $A_j$ for $j\geq 1$. It is then true that $ m_+(N,t)$ exists and lies in $\C^+$ for arbitrary $N\in \R$ if $ t \in\bigcup_{j=1}^N A_j.$ Moreover, we need  $m_j \in \C^+$ such that\label{bp1} \[ \gamma (m_+(t), m_j) <\epsilon ,\] such $m_j$ can be defined as $m_j= m_+(t_j)$ for any fixed $t_j \in A_j.$ By  Lemma \ref{lem1.4}, there is a number $y>0$ such that , for arbitrary Herglotz function $F$, for any Borel subset $S$ of $\R$ and for all $j= 1,2,.......,n$ we have the estimate  \label{bp0}\begin{align} \Big|   \int_{A_j} \omega_{ F(t+iy)}(S)dt - \int_{A_j} \omega_{F(t)}(S)dt  \Big| \leq \epsilon|A_j| .\end{align} We can define $y$ for each value of $j$ ; so $y$ is a function of $j$. However, by taking the minimum value of $y(j)$ as $j$ runs from $1$ to $n$ we my assume $y$ is independent of $j$. Let   $ M_j (N,z)= \frac{u_2(N,z)+\bar{m}_j v_2(N,z)}{u_1(N,z)+\bar{m}_j v_1(N,z)}$ for any $z\in \C^+$.
We shall complete the proof of the theorem by showing that, for $j\geq 1$,\\ (i): $ \int_{A_j}w_{m_{+}(N,t)}(S)dt$ is close to the integral $  \int_{A_j}\omega _{\overline{M_j}(N,t)}(S)dt $ \\ where  $ M_j(N,t)= \frac{u_2(N,t)+\bar{m}_j v_2(N,t)}{u_1(N,t)+ \bar{m}_j v_1(N,t)}$ and that \\(ii):  $\int_{A}\omega_{m_{-}(N,t)}(-S)dt $ is close to the same integral for all $N$ sufficiently large.\\ \\
 \emph{ Proof of (i):} We have \[ m_+(N,t)= \frac{u_2(N,t)+m_+(t) v_2(N,t)}{u_1(N,t)+m_+(t) v_1(N,t).}\] Hence, for fixed $N$ and $t$, the mapping from $m_+(t)$ to $ m_+(N,t)$ is a Mobius transformation with real coefficients and discriminant$u_1v_2-v_1u_2=1.$ and $\gamma$ is invariant under Mobius transformations. Now from \ref{bp1} we see that \[ \gamma \Big( m_+(N,t),\frac{u_2(N,t)+m_j v_2(N,t)}{u_1(N,t)+m_j v_1(N,t)}\Big) \leq \epsilon \text{ for  } j\geq 1 \text {  and  } t \in A_j.\] By equation \ref{hyprdistance} we see that,\[ \Big| \omega_{m_+(N,t)}(S)- \omega_{M_j(N,t)}(S)\Big| \leq \epsilon ,\] and integration with respect to $t$ over $A_j$ gives the estimate

  \label{bp3}\begin{align}\Big| \int_{A_j}\omega_{m_+(N,t)}(S)dt-\int_{A_j}\omega_{\overline{M_j}(N,t)}(S)dt\Big|\leq\epsilon|A_j|.\end{align} This holds for all $j= 1,2,.....n.$\\
  \emph{Proof of (ii):} For $j\geq1$, define the subset $A_j^y$ of $\C^+$, consisting of all $ z\in \C^+ $ of the form $ z = t+iy$, for $ t \in A_j $. Thus $ A_j^y $ is the translation of $ A_j $ by distance $ y $ above the real $z$-axis. Since $A_j$ is bounded, $A_j^y$ is contained in a compact subset of $\C^+$. Hence by Corollary \ref{bpc1} there a positive number $N_0$ such that for $j \geq 1, N \geq N_0$ and $z \in A_j^y $ we have the estimate

\label{bp4}\begin{align} \gamma\Big( -\frac{v_2(N,z)}{v_1(N,z)}, - \frac{u_2(N,z)+\bar{m}_j v_2(N,z)}{u_1(N,z)+\bar{m}_j v_1(N,z)}\Big) \leq \epsilon. \end{align} As in the case of $y$ we may choose $N_0$ to be independent of $j$. Let $m_-(N,z)= -\frac{v_2(N,z)}{v_1(N,z)}$. Following the similar argument to that in the proof of (i), for any $z= t+iy$ we have the estimate
\begin{align*} \Big|\int_{A_j}\omega_{m_-(N,z)}(-S)dt-\int_{A_j}\omega_{-M_j(N,z)}(-S)dt\Big|\leq\epsilon|A_j| , \end{align*} valid for $j\geq1$ and $N\geq N_0.$ Now by Lemma \ref{lem1.4}, equation \ref{bp0} we have,
\[ \Big| \int_{A_j}\omega_{m_-(N,t)}(-S)dt - \int_{A_j}\omega_{-M_j(N,t)}(-S)dt\Big|\leq 3\epsilon|A_j|. \]Now using the identity $ \omega_{-w}(S)= \omega_{\bar{w}}(S)$ \label{bp5}\begin{align} \Big| \int_{A_j}\omega_{m_-(N,t)}(-S)dt-\int_{A_j}\omega_{\overline{M_j}(N,t)}(S)dt\Big| \leq 3\epsilon |A_j|,\end{align} which holds for all $j\geq1$ and $N\geq N_0$ and completes the proof of (ii). Combining the inequalities \ref{bp3} and \ref{bp5} now yields, for $j\geq 1$ and $N\geq N_0$, \label{bp6}\begin{align} \Big| \int_{A_j}\omega_{m_-(N,t)}(-S)dt-\int_{A_j}\omega_{m_+(N,t)}(S)dt\Big| \leq 4\epsilon |A_j|.\end{align} Noting that $A_0$ was chosen such that $|A_0| \leq \epsilon |A|$ we now have for all $ N \geq N_0,$

\begin{align*} & \Big| \int_{A}\omega_{m_-(N,t)}(-S)dt-\int_{A}\omega_{m_+(N,t)}(S)dt\Big|         \\ & \leq  \displaystyle \sum_{j=0}^n  \Big| \int_{A_j}\omega_{m_-(N,t)}(-S)dt-\int_{A_j}\omega_{m_+(N,t)}(S)dt\Big|    \\  & \leq   2|A_0|+4 \epsilon \sum_{j=0}^n |A_j|  \leq\epsilon|A_j| \leq 6 \epsilon |A|. \end{align*} Since $\epsilon $ was arbitrary, the theorem follows.

\emph{ Proof of Theorem \ref{RT}:} The proof is basically same as in \cite{CR1}, however let me sketch:

 Let $\nu \in \omega(\mu).$ Then there exists a sequence $x_n\rightarrow \infty $ such that $ d(S_{x_n}\mu,\nu)\rightarrow 0.$ Then by Lemma \ref{Lem6.2} we have that \[ m_{\pm}(x_n,z)\rightarrow M_{\pm}(z)\,\,\,(n\rightarrow \infty),\]uniformly on compact subset of $\C^+.$ Here $M_{\pm}(z)= m_{\pm}^{\nu}(0,z)$ are the $m$ functions of the whole line Hamiltonian $\nu.$ By Theorem \ref{VD} we see that \[ m_{\pm}(x_n,z)\rightarrow M_{\pm}(z)\,\,\,(n\rightarrow \infty),\] in value distribution. That is \[ \displaystyle \lim_{n\rightarrow\infty} \int_A \omega_{m_{\pm}(x_n,t)}(S)dt = \int_A \omega_{M_{\pm}(t)}(S) dt \] for all Borel sets $A, S \subset \R, |A|<\infty.$ Also by Theorem \ref{Bp} we have \[\int_A \omega_{M_{-}(t)}(-S)dt = \int_A \omega_{M_{+}(t)}(S)dt.\] By Lebesgue differentiation theorem,

 \begin{align} \label{RT1} \omega_{M_{-}(t)}(-S)=  \omega_{M_{+}(t)}(S) \end{align} for $t \in \Sigma_{ac}$ and all intervals $S$ with rational end points. We can also assume that  $M_{\pm}(t)= \lim_{y\rightarrow0+}M(t+iy)$ exists for these $t$. Moreover, if $M_{-}(t)\in\R,$ then, by choosing small intervals about this value for $-S$, we see that $M_+(t)= -M_{-}(t).$ If $ M_{-}(t)\in\C$, then \begin{align*} \omega_{M_{-}(t)}(-S)& = \int_{(-S)}\frac{v}{(t-u)^2+v^2}dt\\ &=-\int_{(S)}\frac{v}{(t+u)^2+v^2}dt \\ &= \omega_{-\overline {M_{-}}(t)}(S) .\end{align*} By \ref{RT1} we get,\begin{align} \label{RT2} M_+(t)= -\overline {M_{-}(t)}. \end{align} In the case when $M_{-}(t)\in\R$ we already have $ M_+(t)= - M_{-}(t).$ So \ref{RT2} holds for almost every $t\in \Sigma_{ac},$ that is $\nu \in \mathcal R(\Sigma_{ac}).$ This completes the proof.

\section{ Relation between a Schrodinger Equation / Jacobi Equation  and a Canonical System }\label{jsc}

\subsection{Reduction of Schrodinger Equation to a Canonical System}

 Let \begin{align}\label{sc}-y''+ V(x)y=zy\end{align}
be a Schrodinger equation. Suppose $u(z,z)$ and $v(x,z)$ are the linearly independent
solutions of \ref{sc}, satisfying some boundary condition $\alpha$ at
$0$. Then $u_0=u(x,0)$ and $v_0=v_0(x,0)$ are the solutions of
$-y''+ V(x)y=0$. Let \[H(x)=\begin{pmatrix}u_0^2 & u_0v_0\\u_0v_0 & v_0^2\\
\end{pmatrix} \] then the Schrodinger equation \ref{sc} is equivalent with the canonical system
 \begin{equation} \label{sac} Jy'= z Hy \end{equation}

 That is if $y$ solves equation  \ref{sc}  then  $U(x,z)= T^{-1}(x)\begin{pmatrix}y(x,z)\\ y'(x,z)\\ \end{pmatrix}$
solves the canonical system \ref{sac}.

  \textbf{Alternative Approach }: Let \begin{align}\label{sc1}-y''+ V(x)y=z^2y \end{align}be
  a Schrodinger equation such that $-\frac{d^2}{dx^2}+V(x)\geq 0$
  and $y(x,z)$ be its solution. Then $y_0= y(x,0)$ be a solution
  of $-y''+ V(x)y=0$. Let $W(x)=\frac{y_0'}{y_0}$ then $W^2(x)+W'(x)=
  V(x)$ so that equation \ref{sc1} becomes\begin{align}\label{sc2}-y''+ (W^2+W')y=z^2y
  .\end{align}Claim that the equation \ref{sc2} is equivalent with
  the  Dirac system \begin{align}\label{ca1} Ju'=\begin{pmatrix}z & W\\W & z\\ \end{pmatrix}u
  .\end{align}If $y$ is a solution of \ref{sc2} then $u =\begin{pmatrix}y\\-\frac{1}{z}(-y'+Wy)
  \end{pmatrix}$ is a solution of \ref{ca1}. Also if $u
  =\begin{pmatrix}u_1\\u_2  \end{pmatrix}$ is a solution of
  \ref{ca1} then $u_1$ is a solution of \ref{sc2}. Next we show
  that the Dirac system \ref{ca1} is equivalent with the Canonical
  System \begin{align}\label{ca2} Ju'(x)=z H(x)u(x),\,\,\,\, H(x)=\begin{pmatrix} e^{2\int_0^x
  W(t)dt}& 0\\0& e^{-2\int_0^x W(t)dt} \end{pmatrix}
  .\end{align}For if $u$ is a solution of \ref{ca1} then $T_0 u$,
  where $T_0 =\begin{pmatrix} e^{-\int_0^x W(t)dt}& 0\\0& e^{\int_0^x W(t)dt}
  \end{pmatrix}$ is a solution of \ref{ca2}.\\

  If we consider a Schrodinger equation of the form,\begin{align}\label{ca6} -y''+(W^2-W')y=z^2y \end{align} then it is equivalent with the Dirac system\begin{align}\label{ca7} Ju'=\begin{pmatrix}z & -W\\-W & z\\ \end{pmatrix}u .\end{align} In other words, if $y$ is a solution of Schrodinger equation \ref{ca6} then $u=\begin{pmatrix}zy\\ y'+Wy \end{pmatrix}$ is a solution of the Dirac system \ref{ca7}. Conversely, if $u=\begin{pmatrix}u_1\\u_2  \end{pmatrix}$ is a solution of the Dirac system \ref{ca7} then $u_1$ is a solution to the Schrodinger equation \ref{ca6}.

 The Dirac system \ref{ca7} is equivalent with the canonical system,
 \begin{align}\label{ca8} Ju'(x)=z H(x)u(x)\end{align} where  $ H(x)=\begin{pmatrix} e^{-2\int_0^x W(t)dt}& 0\\0& e^{2\int_0^x W(t)dt} \end{pmatrix}.$  If $u$ is a solution of the Dirac system \ref{ca7} then $y= T_0u,\,\,\,\, T_0=\begin{pmatrix} e^{\int_0^x W(t)dt}& 0\\0& e^{-\int_0^x W(t)dt} \end{pmatrix} $ is a solution of the canonical system \ref{ca8}. Conversely if $u$ is a solution of the canonical system \ref{ca8} then $T_0^{-1}u$ is a solution of the Dirac system \ref{ca7}.

  \subsection{Reduction of a Jacobi Equation to a canonical
  system.}

    Let a Jacobi equation be \begin{align} \label{Ja1}
  a(n)u(n+1)+a(n-1)u(n)+b(n)u(n)= zu(n) .\end{align}
  This equation
  can be written as \begin{align*} \begin{pmatrix}u(n)\\u(n+1) \end{pmatrix} &= \begin{pmatrix}0 & 1 \\ -\frac{a(n-1)}{a(n)} & \frac{z-b(n)}{a(n)} \end{pmatrix}.
  \begin{pmatrix}u(n-1)\\u(n)\end{pmatrix} \\ &= [B(n) +z A(n)]
  \begin{pmatrix}u(n-1)\\u(n)\end{pmatrix}.\end{align*} Where $ B(n)=\begin{pmatrix}0 & 1\\-\frac{a(n-1)}{a(n)}&\frac{-b(n)}{a(n)} \end{pmatrix}$ and
   $A(n)=\begin{pmatrix}0 & 0\\0& \frac{1}{a(n)} \end{pmatrix}.$ Suppose
  $p(n,z)$ and $q(n,z)$ be the solutions of \ref{Ja1} such that
  $p(0,z)=1, p(1,z)=1$ and $q(0,z)=0, q(1,z)=1.$ So that $p_0(n)=p(n,0)$
  and $q_0(n)= q(n,0)$ be the solutions of equation \ref{Ja1} when
  $z=0.$ Then  \begin{align*} \begin{pmatrix}p_0(n)\\p_0(n+1) \end{pmatrix} = \begin{pmatrix}0 & 1\\-\frac{a(n-1)}{a(n)}&\frac{-b(n)}{a(n)}
  \end{pmatrix}  \begin{pmatrix}p_0(n-1)\\p_0(n)
  \end{pmatrix}.\end{align*}
  (similar expression for $q_0(n)$.) Let $T(n)=\begin{pmatrix}p_0(n-1) & q_0(n-1)\\p_0(n)&q_0(n) \end{pmatrix}
  ,  T(1)= 1.$ Then we have the relation $T(n+1)=B(n)T(n).$
Now define \\ $U(n,z)= T^{-1}(n+1)Y(n,z) , Y(n,z)=
\begin{pmatrix}p_(n-1,z) & q_(n-1,z)\\p_(n,z)&q_(n,z)
\end{pmatrix}.$ Then $U(n,z)$ solves an equation of the form
\begin{align} \label{ca3} J\big(U(n+1,z)-U(n,z)=zH(n)U(n,z)\big)\end{align}
 where $ H(n) =JT^{-1}(n+1)A(n)T(n).$  Suppose for each $n\in\Z$, on $(n,n+1)$, $H$
has the form \begin{align*} H(x)= h(x)P_{\phi}, \hspace{1.cm}
P_{\phi}=
\begin{pmatrix}\cos^2\phi & \sin\phi\cos\phi\\ \sin\phi\cos\phi & \sin^2\phi
\end{pmatrix} \end{align*} for some $\phi \in [0,\pi )$ and some $h\in L_1(n,n+1), h\geq 0.$ (We may choose $h(x)\equiv 1$ on $(n,n+1)$ for each $n\in\Z$) Then the canonical system \ref{ca}
reads \begin{align*}u'(x)= -zh(x)JP_{\phi} u(x). \end{align*}
Since the matrices on the right-hand side commute with one another
for different values of $x,$ the solution is given by
\begin{align*} u(x)= \text{exp} \Big(-z \int_{a}^{x} h(t)dt JP_{\phi} \Big)  u(a). \end{align*}
However, $P_{\phi}JP_{\phi} = 0,$ we see that the exponential
terminates and we get \begin{align} \label{ca4} u(x)=  \Big( 1-z
\int_{a}^{x} h(t)dt JP_{\phi} \Big)  u(a). \end{align} Clearly
equation \ref{ca4} is equivalent with the equation \ref{ca3}.

\subsection{ Relation between Weyl-m functions}

We next observe the relation between the Weyl-m functions for Shrodinger equation and the canonical system \ref{ca}.

\begin{lemma}\label{rm1} For $z\in\C^+,$ let $m_s(z), m_c(z)$ denote the Weyl m-functions corresponding to the Schrodinger equation \ref{sc} and the  canonical system \ref{sac} respectively. Then  $ m_s(z)=m_c(z). $ \end{lemma}

\begin{proof}
Let $T_s(x,z)= \begin{pmatrix}u(x,z) & v(x,z)\\u'(x,z) & v'(x,z) \end{pmatrix}$ and $ T_c(x,z)=\begin{pmatrix}u_1(x,z) & v_1(x,z)\\u_2(x,z) & v_2(x,z) \end{pmatrix}$ are the transfer matrices corresponding to the Schrodinger equation \ref{sc} and the  canonical system \ref{sac} respectively.
Let $ T_0(x)=T_s(x,0)$ then  in \ref{sac}, $H(x)= T_0^* \begin{pmatrix}1 & 0 \\0 & 0 \end{pmatrix}T_0.$ Here $ m_s(z)$ is such that $(1,0)T_s(x,z)\begin{pmatrix}1 \\m_s(z) \end{pmatrix} \in L^2(R_+)$ and $m_c(z)$ is such that $T_c(x,z)\begin{pmatrix}1 \\m_c(z) \end{pmatrix} \in L^2( H, R_+)$. Note that here,  $ T_s(x,z)=T_0(x)T_c(x,z)$

It follows that,   \begin{align*} \displaystyle & \int_0^{\infty} (1,\bar{m}_s)T_s^*(x,z)\begin{pmatrix}1 & 0 \\0 & 0 \end{pmatrix}T_s(x,z)\begin{pmatrix}1 \\m_s(z) \end{pmatrix} dx < \infty  \\  \Rightarrow  & \int_0^{\infty}(1,\bar{m}_s)T_c^*(x,z)T_0^*(x)\begin{pmatrix}1 & 0 \\0 & 0 \end{pmatrix}T_0(x)T_c(x,z) \begin{pmatrix}1 \\m_s(z) \end{pmatrix}dx < \infty \\  \Rightarrow & \int_0^{\infty} (1,\bar{m}_s)T_c^*(x,z)H T_c(x,z)\begin{pmatrix}1 \\ m_s(z) \end{pmatrix}dx <\infty .\end{align*} Since the Weyl-m function $m_c(z)$ is uniquely defined we must have $ m_s(z)=m_c(z). $ \end{proof}

\begin{lemma} For $z\in\C^+,$ let $m_s(z^2), m_c(z)$ denote the Weyl m-functions corresponding to the Schrodinger equation \ref{ca6} and the  canonical system \ref{ca8} respectively. Then  $ m_s(z^2)= z m_c(z). $ \end{lemma}

\begin{proof} Note that, since $ H(x)= \begin{pmatrix} e^{2\int_0^x
  W(t)dt}& 0\\0& e^{-2\int_0^x W(t)dt} \end{pmatrix},  f\in L^2(H,\R_+)$ if and only if \[ \int_0^{\infty}|f_1|^2e^{2\int
  _0^xW(t)dt}dx < \infty  , \,\, \int_0^{\infty}|f_2|^2e^{-2\int
  _0^xW(t)dt}dx <  \infty .\] Let $T_s(x,z^2), T_d(x,z)$ and $T_c(x,z)$ denote the transfer matrices of the Schrodinger equation \ref{sc2}, the Dirac system \ref{ca1} and the canonical system \ref{ca2} respectively. Then,\begin{align*}T_s(x,z^2) & = \begin{pmatrix}u(x,z^2) & v(x,z^2)\\u'(x,z^2) & v'(x,z^2) \end{pmatrix}, \\  T_d(x,z) & =  \begin{pmatrix}u(x,z^2) & zv(x,z^2)\\ \frac{u'(x,z^2)-W(x)u(x,z^2)}{z} & v'(x,z)-W(x)v(x,z),\end{pmatrix}  \\  T_c(x,z) & =T_0T_d(x,z).\end{align*} It follows that \begin{align*}  T_d(x,z)=  \begin{pmatrix}z & 0\\-W & 1\end{pmatrix}T_s(x,z^2)\begin{pmatrix}\frac{1}{z} & 0\\0 & 1\end{pmatrix}.\end{align*} So $ T_d(x,z)=T_0^{-1}T_c(x,z)$ and  \begin{align*} T_s(x,z^2)= \frac{1}{z} \begin{pmatrix}1& 0\\W & z\end{pmatrix}T_d(x,z)\begin{pmatrix}z & 0\\0 & 1\end{pmatrix}.\end{align*}

Now we have,

\begin{align*}\displaystyle &\int_0^{\infty} (1,\bar{m}_c(z))T_c^*(x,z) H(x) T_c(x,z)\begin{pmatrix}1 \\ m_c(z) \end{pmatrix}dx  < \infty \\   \Rightarrow  &  \int_0^{\infty}(1,\bar{m}_c(z))T_c^*(x,z)\Big[T_0^{-1}(x)\begin{pmatrix}1 & 0 \\0 & 0 \end{pmatrix}T_0(x)^{-1}+  \\ & \hspace{1in} T_0(x)^{-1}\begin{pmatrix}0 & 0 \\0 & 1 \end{pmatrix}T_0(x)^{-1}\Big]T_c(x,z) \begin{pmatrix}1 \\m_s(z) \end{pmatrix}dx < \infty .\\  \Rightarrow & \int_0^{\infty} (1,\bar{m}_c(z))T_d^*(x,z) T_0(x)T_0(x)^{-1}\begin{pmatrix}1 & 0 \\0 & 0 \end{pmatrix}. \\ & \hspace{1in} T_0(x)^{-1}T_0(x)T_d(x,z)\begin{pmatrix}1 \\m_c(z) \end{pmatrix} dx < \infty. \\  \Rightarrow & \int_0^{\infty} (1,\bar{m}_c)\begin{pmatrix}\frac{1}{z} & 0\\0 & 1\end{pmatrix}T_s^*(x,z^2) \begin{pmatrix}\bar{z} & W\\0 & 1\end{pmatrix}. \\ & \hspace{1in} \begin{pmatrix}0 & 0 \\0 & 1 \end{pmatrix}  \begin{pmatrix}z & 0\\-W & 1\end{pmatrix}T_s(x,z^2)\begin{pmatrix}\frac{1}{z} & 0\\0 & 1\end{pmatrix}\begin{pmatrix}1 \\ m_c(z) \end{pmatrix}dx <\infty .\\   \Rightarrow & \int_0^{\infty} \begin{pmatrix} \frac{1}{z}&\bar{m}_c\end{pmatrix}T_s^*(x,z^2) \begin{pmatrix}0 & 0 \\0 & 1 \end{pmatrix}
T_s(x,z^2)\begin{pmatrix}\frac{1}{z} \\ m_c(z) \end{pmatrix}dx <\infty .\end{align*} Since the Weyl-m function $m_c(z)$ is uniquely defined we must have \[ m_s(z^2)=zm_c(z).\] \end{proof}

Suppose \[ H_+ = \begin{pmatrix} e^{2\int_0^x W(t)dt}& 0\\0& e^{-2\int_0^x W(t)dt} \end{pmatrix},\,\,\, H_- = \begin{pmatrix} e^{-2\int_0^x W(t)dt}& 0\\0& e^{2\int_0^x W(t)dt}\end{pmatrix}\] in the canonical system \ref{ca2} and \ref{ca8} respectively. The following lemma shows the relation between their Weyl-m functions.

\begin{lemma}If $m_{c_+}$  and $m_{c_-} $ are the Weyl-m function corresponding to the canonical system \ref{ca2} and \ref{ca8} respectively then $ m_{c_+} = \frac{-1}{m_{c_-}}$.\end{lemma}

\begin{proof} Notice that $ -JH_+ J  = H_-.$ Here  $ u$ is a solution of $ Ju'=zH_+ u $ if and only if  $ Ju$ is a solution of $ Ju'=zH_- u.$ Let $ T_{c_+}(x) $ and $T_{c_-}(x)$ be the transfer matrices  and  $ m_{c_+}$ and $m_{c_-}$ are the Weyl-m functions of the canonical systems with the Hamiltonians $ H_+$ and $H_-$ respectively. Then $ T_{c_-}(x) = -JT_{c_+}(x)J$  and  \begin{align*} &  \int_0^{\infty} (1, \bar{m}_{c_-})T^*_{c_-}(x)H_-T_{c_-}(x) \begin{pmatrix}1 \\m_{c_-} \end{pmatrix}dx < \infty  \\ \Rightarrow & \int_0^{\infty} (1, \bar{m}_{c_-})(-JT_{c_+}(x)J)^* H_- ( -JT_{c_+}(x)J ) \begin{pmatrix}1 \\m_{c_-} \end{pmatrix}dx < \infty   \\ \Rightarrow & \int_0^{\infty} \begin{pmatrix} 1, & \frac{-1}{ \bar{m}_{c_-}}\end{pmatrix} T^*_{c_+}(x)H_+ T_{c_+}(x) \begin{pmatrix} 1 \\ \frac{-1}{ \bar{m}_{c_-}} \end{pmatrix}dx < \infty  .\end{align*}  Since $ m_{c_+}$ is the unique coefficient such that \begin{align*}\int_0^{\infty} (1, \bar{m}_{c_+})T^*_{c_+}(x)H_+T_{c_+}(x) \begin{pmatrix} 1 \\ m_{c_+} \end{pmatrix}dx < \infty   \end{align*}  we have $ m_{c_+} = \frac{-1}{m_{c_-}}. $ \end{proof}

\begin{theorem}Let $\omega(V)$ and $\omega(H)$ are the $\omega$-limit set corresponding to a Schrodinger equation \ref{sc} and its canonical system \ref{sac} respectively. Then if $W \in \omega(V)$  then $ K \in \omega(H)$ where $K$ is the Hamiltonian corresponding to a canonical system  of the Schrodinger equation with $W$ as potential. Conversely, if $K \in \omega(H)$ then $K$ is a Hamiltonian for a canonical system of a Schrodinger equation for some potential $W\in \omega(V).$ \end{theorem}

\begin{proof} Suppose $W \in \omega(V)$ then by definition of $\omega-$ limit set there exists a sequence $x_n\rightarrow \infty$ such that $V(x+x_n)\rightarrow W. $ Then the corresponding Weyl m-functions also converge, ie $ m_s^{V_n}(z)\rightarrow m_s^{W}(z).$ Let $H_n$ be the Hamiltonian of the canonical system obtained from the Schrodinger equation with the potential $V(x+x_n)$ then $H_n= H(x+x_n)$ then by Lemma \ref{rm1} $  m_s^{V_n}(z)=  m_c^{H_n}(z).$ and $ m_s^{W}(z)= m_c^{H}(z).$ Now apply the change of variable by \ref{cv} and obtain $ \widetilde{H}_n$ and the corresponding m-function is $m_c^{\widetilde{H}_n}(z).$ After the change of variable the corresponding Weyl m-functions are the same up to the change of the point of boundary condition. So the convergence of $ m_s^{V_n}(z)=m_c^{H_n}(z)$ implies the convergence of $ m_c^{\widetilde{H}_n}(z). $ It follows that $ m_c^{\widetilde{H}_n}(z) \rightarrow  m_s^{W}(z) . $ But by Lemma \ref{Lem6.2} $m_s^{W}(z)= m_c^{\widetilde {H}}(z)$ where $m_c^{\widetilde H}(z)$ is the Weyl m-function for some Hamiltonian $\widetilde {H}.$ It follows that $ m_c^{\widetilde{H}_n}(z) \rightarrow m_c^{\widetilde {H}}(z).$  Again by Lemma \ref{Lem6.2} we get $\widetilde{H}_n \rightarrow \widetilde {H}  $ using the change of variable on the canonical system with Hamiltonian $\widetilde {H}$ we obtain a Hamiltonian $ K $ such that $m_c^{\widetilde {H}}(z)= m_c^{K}(z)$ up to the change of point of boundary condition. It follows that $H_n \rightarrow K$ and so $K \in \omega(H) .$ Converse is similar.
  \end{proof}

\providecommand{\bysame}{\leavevmode\hbox to3em{\hrulefill}\thinspace}
\providecommand{\MR}{\relax\ifhmode\unskip\space\fi MR }
\providecommand{\MRhref}[2]{%
  \href{http://www.ams.org/mathscinet-getitem?mr=#1}{#2}
}
\providecommand{\href}[2]{#2}

\vspace{0.5in}Department of Mathematics,\\ University of Oklahoma, Norman, OK, 73019 \\ E-mail Address: kacharya@math.ou.edu

\end{document}